\documentclass{llncs}
\usepackage{llncsdoc}
\usepackage{amssymb}
 
\title{Proof Search Algorithm in Pure Logical Framework}
\author{Dmitry Vlasov\inst{1}}
\institute{Sobolev Institute of Mathematics, SB RAS \email{vlasov@math.nsc.ru} 
\footnote{The research is partially supported by Russian Foundation for Basic Researches (Grant No. 17-01-00531)}}

\begin{document}

\maketitle

\begin{abstract}
By a pure logical framework we mean a framework which does not rely on any
particular formal calculus. For example, Metamath \cite{Metamath} is an instance of 
a pure logical framework. Another example is the Russell system
(https://github.com/dmitry-vlasov/russell). In this paper, we describe the proof 
search algorithm used in Russell. The algorithm is proved to be correct and complete, i.e.
it gives only valid proofs and any valid proof can be found (up to a
substitution) by the proposed algorithm.
\end{abstract}

\section{Introduction}

Historically, there are several approaches to automated reasoning for logical frameworks.
The most popular is LCF \cite{LCF}, where the proofs are generated by programs (tactics), and a user
is responsible for programming these tactics and using them in the process of proving a
statement \cite{J_Harrison} - such systems as HOL, Isabelle, etc. use this approach. 
The other approach was invented and was explored by S. Maslov \cite{Maslov64} - so called 
inverse method. It was found, that this method may be applied to various calculi, which
satisfy some good properties \cite{Maslov68} and about that time attempts to apply this 
method to computer programs \cite{comp_deduct} were held.

The algorithm, which will be discussed in this paper has nothing in common with LCF
methodology, but is shares specific features with the inverse method. Although, strictly
speaking, this algorithm is of bottom-up kind (see classification of
algorithms in \cite{J_Harrison}, p.172), it has a top-down component, which, in turn,
shares some common features with the inverse method. The other method, which has something 
common with the proposed algorithm is Prawitz method \cite{Prawitz}. Namely the idea,
which is common in all three methods: Prawitz, inverse and proposed in this paper, is
that we should look for a special substitution, which will make different parts of 
inference (whatever we mean by this word) compatible with each other, while going 'downwards', 
from premises to goal, and this substitution may be computed. The resolution method
\cite{Resolution} also uses analogical idea (namely unification), but it tries to unify 
positive and negative entries of a proposition, therefore searching for an inconsistency,
instead of compatibility.

And what is different between 
top-down approach, discussed in \cite{J_Harrison}, and method presented in this paper,
is that here top-down pass has local character, i.e. it affects only a currently
observed inference transition. This helps to make it efficient, although its efficient 
implementation is far from trivial.

\section{Inference in a Pure Logical Framework}
\subsection{Language and Deductive System}
Let's consider a definition of some deductive system $\mathcal{D}$.
First of all, let's fix a 
context free unambiguous grammar $G$ for a language of expressions $L(G)$ with set of non-terminals $N$. 
For each $n\in N$ let's designate as $G_n$ a grammar, obtained from $G$
by changing the start terminal to $n$. We need these grammars because the formalized
calculus may exploit expressions from different $G_n$ languages. To show the
fact, that an expression $e$ belongs to $G_n$ language we'll use notation $e:n$.
We suppose, that for each non-terminal $n$ there's an infinite set of terminal symbols $v$,
such that the rule $n\rightarrow v$ is in $G$, so that $v:n$. We'll call symbol $v$ a \textit{variable} of type $n$.

The \textit{assertion} has a form $a=\frac{p_1,\ldots, p_n}{p_0}$ where $p_1,\ldots,p_n$ are 
\textit{premises} and $p_0$ is a \textit{proposition} of $a$. The \textit{deductive system} 
$\mathcal{D}$ is a pair:  $\mathcal{D}=(G,A)$ where $G$ is a grammar of expressions and $A$ 
is a set of axiomatic assertions of $\mathcal{D}$.

\subsection{Unification}
A mapping $v:n \stackrel{\theta}{\mapsto} e:m$ is called a \textit{substitution}, iff the rule
$n\rightarrow m$ is derivable in grammar $G$. This means that in grammar $G$ you may substitute 
non-terminal $n$ with non-terminal $m$, and, therefore, with the expression $e$. The classical
example of such relation is class-set relation in set theory with classes: you can substitute
a class-typed variable with an expressions of type set (and you cannot do it in the other direction).

Application of a substitution $\theta$ to an expression $e$ is straightforward and is designated as $\theta(e)$.
Note, that application of substitution demands a conformity of variable types. Further
we will assume, that all application of substitution are correct.
By definition of substitution, if $e$ is an expression from $L(G)$, and $\theta$ is applicable 
to $e$ then $\theta(e)$ also stays in $L(G)$ - this follows from the context-freeness of $G$. 
The application of substitution to a complex object 
(like assertion or proof) is understood component-wise. 
There is a natural notion of \textit{composition} of substitutions, which we will
designate as $(\theta\circ\eta)(e)=\theta(\eta(e))$.

An expression $e_1$ \textit{unifies} with an expression $e_2$, iff there is a substitution $\theta$
such that $e_2=\theta(e_1)$. Note, that here unification is non-symmetric, e.g. it may happen that
$e_1$ is unifiable with $e_2$, but $e_2$ does not unify with $e_1$. The reason 
for this is in the directed nature of inference in our case. The consequence of such non-symmetry is
that if a unifier is exists, it is unique. 

\subsection{Inference}
A \textit{proof tree} in a deduction system $\mathcal{D}=(G,A)$ is a tree, which nodes are labeled: 
of odd depth with assertions from $A$ (a-nodes) and of even depth with expressions from $L(G)$ (e-nodes) and 
any e-node has at most one descendant a-node. Also we demand that leafs must be only e-nodes and 
a-nodes are labeled with substitutions, so actually a label of an a-node is a pair: $(a, \theta)$. Note, that
here we don't demand validity: proof trees may be invalid.

Two proof trees $T_1$ and $T_2$ are \textit{congruent}, iff their graph structures are isomorphic, corresponding 
a-nodes have the same assertions as labels, and no restrictions on e-nodes. Let's designate this relation as $T_1\sim T_2$.

Given two proof trees $T_1$ and $T_2$, we say that $T_1$ is \textit{more general} then $T_2$ (respectively,
$T_2$ is \textit{less general} then $T_1$), iff there exist such substitution $\delta$,
that $T_2 = \delta(T_1)$ (recall, that application of substitution to complex structures is
component-wise). Let's designate this relation as $T_1\succeq T_2$. From the definition it immediately
follows, that $T_1\succeq T_2$ implies that $T_1\sim T_2$.

A proof tree is a \textit{proof}, iff for any transition from e-nodes $e_1,\ldots,e_m$ 
via a-node $(a, \theta)$ to e-node $e_0$, we have that $\frac{e_1,\ldots.e_n}{e_0}$ 
is equal to $\theta(a)$. We'll call the unifier $\theta$ a \textit{witness} of this transition.

A proof tree is a \textit{proof of a statement} $s=\frac{p_1,\ldots, p_n}{p_0}$
iff it is a proof, its root is equal to $p_0$, and each leaf $e$ coincides with some premise $p_i$, $i\leq n$. 

\begin{lemma}[Monotonicity]
Let $\pi$ be a proof of an assertion $a=\frac{p_1,\ldots,p_n}{p_0}$ and $\theta$ - any substitution of variables,
which occur in $a$. Then $\theta(\pi)$ would be the proof of $\theta(a)$. 
\end{lemma}
\begin{proof}
It is sufficient to notice, that application of a substitution to each proof transition keeps unifiability: if $e_0$
is deduced from $e_1,\ldots,e_n$ in the proof $\pi$ with assertion $b$ and unifier $\eta$ is its witness, then
$\theta(e_0)$ is deduced from $\theta(e_1),\ldots,\theta(e_n)$ with the same $a$ and unifier $\theta\circ\eta$ 
is its witness. 
$\square$
\end{proof}

The notion of a deductive system, presented here is very close to the notion of canonical deductive system
given by Post \cite{Post}. The difference is in the treatment of expressions: Post's canonical system
don't have any restrictions on the language: we may substitute variables with any words. Here
we restrict the language to context-free unambiguous class, so that efficient unification algorithms are
possible. Also, monotonicity property, which is demanded for the proof of correctness, is proved only for 
the context-free grammars. Unambiguity of grammar is necessary for the uniqueness of a unifier.

\section{Proof Search Algorithm}
Suppose that we have a deductive system $\mathcal{D}$ and we want to answer the question:
is some particular statement $s$ provable in $\mathcal{D}$? Let's note, that there's no 
symmetry in asking for provability and non-provability is in general case, because it is 
possible to show provability by giving actual proof and checking it, but it is not possible
to assert that something is not provable - we might not have a negation in the deductive system
(once more let's stress here that we are speaking about the general case, for particular decidable 
calculi this is wrong).

From the general considerations, while searching for a proof for the statement $s=\frac{p_1,\ldots, p_n}{p_0}$ 
we may follow different strategies:
\begin{enumerate}
\item start with premises $p_1,\ldots,p_n$ make various inferences and try to obtain the goal $p_0$ (downwards approach)
\item start with the goal $p_0$ look for all possible ways how it can be obtained in $\mathcal{D}$, get 
the sub-goals $q_1,\ldots,q_m$ and do the same for them, until we come up to premises (upwards approach)
\end{enumerate}

The outer loop of the algorithm uses the second variant - upwards search, from goal to premises, but inside of it there is a top-down
loop. So, in a very general sense, the proposed algorithm uses both modes of traversal: bottom-up and top-down, 
but they are not equal and play different roles. Namely, upwards pass is a traversal of possible variants to derive a goal,
while downwards pass is a quest for valid consequences of premises which uses the structure of a tree, which is built during
the upwards pass. 
When we reach the root on the downwards pass,
then the considered statement is proved. Summarizing the above, the proposed algorithm uses a mixture of top-down 
and bottom-up strategies.

\subsection{Proof Variant Tree}
The proof search algorithm essentially is building a tree of proof variants - so called \textit{proof variant tree} 
(PVT). The completeness of the algorithm is guaranteed by the completeness of the tree of variants. The PVT nodes 
are marked with expressions (nodes of even depth) and assertions (nodes with odd depth) - just like proof trees. 
The variables of e-nodes are marked up with replaceable/non-replaceable flags. It is necessary
because some variables are passed from the statement, which is being proved, so they cannot be replaced or
modified, therefore they are marked as non-replaceable; while others come from internal expressions of a proof
and may be substituted with arbitrary expressions. The starting point of the tree building algorithm is a goal expression $p_0$, 
all of its variables are marked as non-replaceable and these variables will stay non-replaceable while
tracing further into the PVT.

Given a node of even depth, which is marked up with an expression $e$, we fork it out with nodes, marked up 
with all assertions $\{a_1,\ldots, a_n\}$ which propositions unify with $e$ with some unifier $\theta$. 
In turn, for each of odd-depth node $a=\frac{q_1,\ldots,q_n}{q_0}$ and appropriate substitution $\theta$, 
its premises $\theta(q_1),\ldots,\theta(q_n)$ form a set of direct descendants of $a$. In some cases there 
may be a collision of variable names at this step, so to avoid it we'll accept an agreement that while
unifying $a$ with $e$, we'll replace all free variables of $a$ with a fresh ones. The binary graph relation 
of precedence in the PVT is designated as $n\succ m$: here $n$ is a direct descendant of $m$.

For any subtree of an PVT we say that it is a \textit{proof variant}, iff any e-node in it has at most
one descendant. Any proof variant $v$ immediately generates a proof tree $\pi(v)$, when we remove
all unrelated data from it. 

\subsection{Substitution Proof Tree}

The nodes in PVT are marked not only by the expressions and assertions. Each node $n$ in PVT
has a set of its \textit{substitution proof trees} (SPT), which is designated as $s(n)$. 
Substitution proof tree $T$ is a proof tree, which nodes are labeled with the nodes of PVT and 
substitutions. The substitution of a root node will be addressed as $\theta(T)$.

Initially, when created, the set of SPT for any node is empty. Let's consider some just created expression 
PVT node $e$. We look at the premises $p_1,\ldots,p_n$ of a statement, which is proved. If some $p_i$ of 
these premises unifies with $e$ (note, that here variables in $e$ are also non-replaceable!), then $e$ is 
trivially provable from $p_i$. So we place the one-node SPT, constructed from the unifier and current 
PVT node, into the set of SPT for this node.

If we find a new SPT node for some expression node, then we try to shift it a step down to the root. 
For this purpose we test all of its siblings (they correspond to the premises of some assertion) for being also proved 
(i.e. the set of SPT is non-empty). If we find, that all siblings of the node are proven, 
we can try to find a SPT node its ancestor. To do it we need a concept of unification of substitutions.

\subsubsection{Unification of Substitutions.}

Given a set of substitutions $\Xi=\{\theta_1,\ldots,\theta_n\}$, we say that a substitution $\delta$ is 
\textit{a unifier} for $\Xi$, iff for all $i,j\leq n$ we have $\delta\circ\theta_i=\delta\circ\theta_j$. 
Here $\circ$ is a composition of substitutions. Unificator $\delta$ is called \textit{most general}, iff 
for any other unifier $\eta$ for the set $\Xi$, there is such $\eta^{\prime}$ that 
$ \eta = \eta^{\prime}\circ\delta $
For each set $\Xi$, if a unifier for $\Xi$ exists, there is a unique 
up to the variable renaming most general unifier, which we will designate as $\mathrm{mgu}(\Xi)$. And the 
common substitution $\mathrm{mgu}(\Xi)\circ\theta_i$ we will designate as $\mathrm{com}(\Xi)$

\subsubsection{Building SPT for Assertion Nodes.}

So, imagine that we have some a-node $a$ in the proof variant tree, and all of its direct descendants
$e_1,\ldots,e_2$ have non-empty sets of SPT $s(e_1),\ldots,s(e_n)$. Then for any tuple of SPT 
$T_1\in s(e_1),\ldots,T_n\in s(e_n)$, if the set of substitutions $\{\theta(T_1),\ldots,\theta(T_n)\}$
is unifiable with $\delta=\mathrm{mgu}(\theta(T_1),\ldots,\theta(T_n))$, then a new SPT $T_0$ with 
substitution $\theta=\mathrm{com}(\theta(T_1),\ldots,\theta(T_n))$ 
is added to $s(a)$. The tree of $T_0$ is obtained as: 
\[   T_0(T_1,\ldots, T_n) = \frac{\delta(T_1)\ \ldots\ \delta(T_n)}{(\theta, a)} \]
Note, that obtained here unifier $\delta$ propagates through the whole trees $T_i$ (applies to
all of its components: expressions and substitutions).
Also, non-replaceable variables cannot be substituted with at this step. So, the set of all SPT for 
the node $a$ will be:
\[s(a) = \{ T_0(T_1,\ldots, T_n) | T_1\in s(e_1), \ldots, T_n\in s(e_n), \exists \mathrm{mgu}(\theta(T_1),\ldots,\theta(T_n)) \}\]

\subsubsection{Building SPT for Expression Nodes.}

The set of SPT for expression node $e$ is updated with update of SPT set of any of its descendants.
For an expression node $e$, if one of its descendants is updated with the SPT $s$, 
then the set of proof-substitutions for $e$ is also updated with a new node, which only descendant is $s$ 
and substitution coincides with the substitution of a descendant:
\[ T_0(T_1) = \frac{T_1}{(\theta, e)} \]
and a set of all SPT for the node $e$ will be:
\[s(e) = \bigcup_{a\succ e}\{T_0(T_1) | T_1\in s(a)\} \]

\begin{lemma}
Each substitution proof tree $T$  defines a unique proof variant $\pi(T)$.
\end{lemma}
\begin{proof}
By induction on the depth of $T$. The base is trivial: when we unify some expression with a premise of
a proven assertion, it clearly generates a proof variant. 
Step of induction comes from the definition of SPT for e-nodes: each SPT e-node has at most one direct
descendant.
$\square$
\end{proof}

\begin{theorem}[Correctness]
For any substitution proof tree $T$ with root $(\theta, e)$ the proof tree $\pi(T)$ is a proof of $\theta(e)$.
\end{theorem}
\begin{proof}
Let's prove it by induction on the depth of $T$. 
The base of induction is trivial: we have no obligation on leafs except for them to be e-nodes.

Let's assume that for some assertion node $a=\frac{q_1,\ldots,q_m}{q_0}$ it has a SPT node $T_0$ with
substitution $\theta_0$ and $T_1,\ldots,T_m$ are direct descendants of $T_0$. Let $e_0$ be a 
unique ancestor of $a$ in the PVT. Then, by definition, 
for substitutions $\theta_1,\ldots,\theta_m$, corresponding to $T_1,\ldots,T_m$ we have that the set
$\{\theta_1,\ldots,\theta_m\}$ has a unifier $\delta$. By induction, 
all $T_i$ induce proofs $\pi_1(T_1),\ldots,\pi_m(T_m)$ for expressions $\theta_1(e_1),\ldots,\theta_m(e_m)$. 
By the definition of unifier of substitutions, for all $i,j\leq m$ we have
\[\delta\circ\theta_i=\delta\circ\theta_j=\theta_0\]
Also there is a unifier $\eta$ such that $\eta(q_0)=e_0$ and $\eta(q_j)=e_j$ for all $j\leq m$.
Let's consider a substitution $\eta^{\prime}=\theta_0\circ\eta$:
\[\eta^{\prime}(q_0)=(\theta_0\circ\eta)(q_0)=\theta_0(\eta(q_0))=\theta_0(e_0)\]
\[\eta^{\prime}(q_i)=(\theta_0\circ\eta)(q_i)=\theta_0(\eta(q_i))=(\delta\circ\theta_i)(\eta(q_i))=
(\delta\circ\theta_i)(e_i)=\delta(\theta_i(e_i))\]
By monotonicity lemma $\delta(\pi(s_i))$ will be a proof of $\delta(\theta_i(e_i))$, so $\eta^{\prime}$ is a
witness for the observed transition in the proof. 
$\square$
\end{proof}

\begin{theorem}[Generality]
If the algorithm finds a substitution proof tree $T$ for the root of PVT for some assertion $a=\frac{p_1,\ldots,p_n}{p_0}$, 
then $\pi(T)$ is more general then any proof $\pi^{\prime}$ of $a$, congruent to $\pi(T)$.
\end{theorem}
\begin{proof}
By the correctness lemma we have that $\pi(T)$ is a proof. By the construction of leaf nodes of $T$,
for each leaf node $e$ from $T$ there is a premise $p_i$ such that $e=p_i$, and the root of the tree $T$
is $(\emptyset, p_0)$, because all variables in $p_0$ are fixed and cannot be substituted with. So
$\pi(T)$ is the proof of $a$.

Now let's check that $\pi(T)$ is a most general. 
Let $\pi^{\prime}$ be another proof of $a$, congruent to $\pi=\pi(T)$.
We prove, that $\pi(T)\succeq \pi^{\prime}$ by induction on the depth of $T$. 
The base of induction is obvious, because leaf nodes of $T$ do not have replaceable variables, 
therefore they are the same for $\pi$ and $\pi^{\prime}$. 
The step of induction. Let's consider some transition in proofs $\pi$ and $\pi^{\prime}$ with
assertion $b=\frac{q_1,\ldots,q_m}{q_0}$ and corresponding expressions 
$e_0, e_1, \ldots, e_m$ and $e^{\prime}_0, e^{\prime}_1, \ldots, e^{\prime}_m$ from $\pi$ and
$\pi^{\prime}$ accordingly. Let $\theta$ and $\theta^{\prime}$ be substitutions, which give $e_i$ 
and $e^{\prime}_i$ from $q_i$ correspondingly. 

By induction, for each $1\leq i \leq 0$ there is a substitution $\varepsilon_i$ such that 
$e^{\prime}_i = \varepsilon_i(e_i)$. Taking copies of the corresponding subtrees with
disjoint sets of variables, we can assume, that all of these substitutions the same,
i.e. $e^{\prime}_i = \varepsilon(e_i)$. By construction of SPT, if we consider all direct SPT-descendants 
$T_1,\ldots, T_m$ of the root of $T$, and their root expressions $e^{\prime\prime}_i$, then 
for the corresponding substitutions $\theta_i$ we'll have a most general unifier
$\delta=\mathrm{mgu}(\theta_1,\ldots,\theta_m)$ and $\theta_0=\mathrm{com}(\theta_1,\ldots,\theta_m)$. 

Now let's write a chain of equations for all $0<i\leq m$:
\[ \theta^{\prime}(q_i)=e^{\prime}_i=\varepsilon(e_i)=\varepsilon(\theta_i(e^{\prime\prime}_i))=
   \varepsilon(\theta_i(\theta(q_i)))= (\varepsilon\circ\theta_i\circ\theta)(q_i)\]

From here we conclude, that $\theta^{\prime}=\varepsilon\circ\theta_i\circ\theta$. Then

\[\varepsilon\circ\theta_i=\varepsilon\circ\theta_j\mbox{,\ \ \  for all }0<i,j\leq m\]
 
Now let's 
recall, that $\delta$ is a most general unifier for $\theta_1,\ldots,\theta_m$, so there 
exists such $\varepsilon^{\prime}$ that $\varepsilon^{\prime}\circ\delta\circ\theta_i=\varepsilon\circ\theta_i$ 
for all $0<i\leq m$, and we get
\[ \theta^{\prime}=\varepsilon^{\prime}\circ\underbrace{\delta\circ\theta_i}_{\theta_0}\circ\theta_i\]
so \[ \theta^{\prime}=\varepsilon^{\prime}\circ\theta_0\circ\theta \]
The only thing, which is left to see that the statement of the induction step holds, is to notice
that $e^{\prime}_0=\theta^{\prime}(p_0)=\varepsilon^{\prime}(\theta_0(\theta(p_0)))=
\varepsilon^{\prime}(e_0))$. $\square$
\end{proof}

\begin{theorem}[Completeness]
If a statement $\frac{p_1,\ldots, p_n}{p_0}$ is has a proof $\pi$, then the set of SPT for the 
root $p_0$ at some moment of building PVT will contain some tree $T$ such that $\pi=\theta(\pi(T))$, 
for some substitution $\theta$.
\end{theorem}
\begin{proof}
By previous theorem it is sufficient to show, that at some moment, the set $s(p_0)$ will contain a
SPT, congruent to $\pi$. But it is clear from the character of the algorithm: at each step
of expansion of PVT, we use \emph{all possible} variants of expansion (limited by demand of unification), 
so at some moment we'll obtain all nodes, corresponding to the proof $\pi$.
$\square$
\end{proof}

\begin{corollary}
The set of all provable assertions in any pure deductive system is computably enumerable.
\end{corollary}

As it already was mentioned, the algorithm has two different aspects: bottom-up and top-down.
The bottom-up procedure (building of a PVT) is quite straightforward. The other one, 
top-down (building SPT's), is more sophisticated, and unification of substitutions is a crucial part of it.
The inverse method, mentioned above, also uses analogical procedure, but, surprisingly,
the unification of substitutions (or, as it called in \cite{Chang} a \textit{combination} of substitutions) 
is not stressed as the main operation, but, rather, another complex transformations of formulas
are considered not less important.

\section{Conclusion and Future Work}
The algorithm, presented in this paper, is the most general proof search algorithm, which one may ever
hope to elaborate. The generality of the underlying formal system in almost maximal, because in comparison 
with the general notion of Post canonical system is restricted only by the language: it has to be context 
free and unambiguous. There's no other
constraints like subformula property, which is vital for the inverse method. And the restriction of grammar
to context-free and unambiguous class is natural: if you don't impose it, then there immediately arises a question
about unification algorithms for the language which is used.

The other good thing about the method, presented in this paper, is that it is completely ready for use out-of-the-box, 
and you don't need to 'cook' \cite{Voronkov} 
a considered logic in order to use it - just write down expression language, axioms and inference rules and
you may feed the assertion of interest to a prover engine, which, in theory, will find a proof (if it exists).

What is left out of scope of this article is unification problem. From the algorithm description it is
clear, that efficient unification algorithms are vital for the implementation of this method. And efficient
unification of an expression with a (potentially huge) set of assertions is usualy done with indexing 
\cite{Voronkov-indexing} and is not trivial. 
The algorithm of unification for substitutions is even more complex and challenging. Experiments on proving a rather simple 
statement in classical propositional Hilbert-style logic showed, that the number of SPT trees may grow extremely fast.
Just to feel the scale of this problem imagine, that we have an assertion with 5 premises (common case in Metamath theorem
base), each of which has a non-empty set of SPT, having, for example 1, 10, 100, 10 and 100 elements respectively. 
Then we need to check $10^6$ substitution tuples for unification. Fortunately, it is experimentally found, 
that almost all of this tuples do not unify, so we'll end up with something like ~500 (or even 0) of solutions, 
but still, checking all of these $10^6$ variants consequently is not affordable in practice. Efficient algorithm for such
massive substitution unification was developed, but it needs a thorough analysis and separate research.

The other thing, which is intentionally missed in this paper, is treatment of proper substitutions for 
disjointed variables. Classical predicate calculus has special restrictions on substitutions, which may be
applied to specific rules of inference (like introduction of $\forall$-quantifier). In Metamath such restrictions
are simplified, but still are essentially a restriction on application of particular substitutions. Addition 
of such restrictions doesn't change the general scheme of algorithm, the only thing, which is necessary 
to track during traversing of PVT are these restrictions, which are not difficult to check. So, in order to 
keep text more simple and clear we decided to skip this details.

The problematic part of practical implementation of such method is computational complexity. The strong side of this 
method is its universality and ability to apply to the wide variety of calculi. And, as always, this universality
causes problems. For example, we cannot rely on good properties of a considered deductive system:
it may have a cut-like rule(s), no subformula property, etc. In practice this leads to the enormous
growth of a search space while searching for a proof. The only way to cope with such combinatorial explosion is 
to use smart heuristics, which will lead search in the right direction.
The author's strong belief is that the advanced methods of machine learning, based on the analysis
of an already formalized proofs, may help to develop such methods.


\begin{thebibliography}{}  


\bibitem{Chang} 
\textit{Chang C.-L., Lee R. C.-T.} 
Symbolic Logic and Mechanical Theorem Proving. Academic Press, (1973)


\bibitem{comp_deduct} 
\textit{Davydov V., Maslov S., Mints G., Orevkov V. and Slissenko A.} 
A computer algorithm of establisihing deducibility based on the inverse method (in Russian),
Zapiski Nauchnyh Seminarov LOMI16, (1969), pages 8 -- 19.


\bibitem{Voronkov} 
\textit{Degtyarev A., Voronkov A.} 
Handbook of Automated Reasoning, chapter 4: The Inverse Method, vol.1, MIT press, (2001), pages 180 -- 272.

\bibitem{Voronkov-indexing} 
\textit{Degtyarev A., Voronkov A.} 
Handbook of Automated Reasoning, chapter 26: Term Indexing, vol.2, MIT press, (2001), pages 1855 -- 1962.

\bibitem{LCF} 
\textit{Gordon, M. J. C.}  
Representing a logic in the LCF metalanguage. 
In Néel, D. (ed.), Tools and Notions for Program Construction: an Advanced Course,
pages 163 -- 185. Cambridge University Press. (1982)

\bibitem{J_Harrison} 
\textit{Harrison J.}
Handbook of Practical Logic and Automated Reasoning, 
Cambridge University Press (2009).

\bibitem{Maslov64} 
\textit{Maslov, S.} 
An inverse method of establishing deducibility in classical predicate calculus. 
Doklady Akademii Nauk (1964), 159, pages 17 -- 20.

\bibitem{Maslov68} 
\textit{Maslov S.} 
The inverse method of establishing deducibility of logical calculi (in Russian),
in 'Collected Works of MIAN', Vol. 98, Moscow, (1968) pages 26 -- 87.


\bibitem{Metamath} 
\textit{Megill, N.} 
Metamath: A Computer Language for Pure Mathematics,
Lulu Press, Morrisville, North Carolina, (2007)

\bibitem{Post} 
\textit{Post E.} 
Formal Reductions of the General Combinatorial Decision Problem, 
American Journal of Mathematics 65 (2), (1943) pages 197 -- 215.

\bibitem{Prawitz} 
\textit{Prawitz, D., Prawitz, H. and Voghera, N.} 
A mechanical proof procedure and its realization in an electronic computer. 
Journal of the ACM , 7, (1960), pages 102 -- 128.


\bibitem{Resolution} 
\textit{Robinson J. A.}  
A machine-oriented logic based on the resolution principle.
Journal of the ACM, (1965) 12, pages 23 -- 41.


\end{thebibliography}
\end{document}